\documentclass[11pt]{article}
\usepackage{amsmath,amsthm,amssymb,amsfonts,mathrsfs,color}
\usepackage{a4wide}
\usepackage{verbatim}

\numberwithin{equation}{section}

\begin{document}

\newtheorem{theorem}{Theorem}%[section]
\newtheorem{proposition}{Proposition}%[section]
\newtheorem{lemma}{Lemma}%[section]
\newtheorem{corollary}{Corollary}%[section]%%
\newtheorem{definition}{Definition}%[section]
\newtheorem{remark}{Remark}%[section]
\newcommand{\tex}{\textstyle}
%\numberwithin{equation}{section} \numberwithin{theorem}{section}
%\numberwithin{proposition}{section} \numberwithin{lemma}{section}
%\numberwithin{corollary}{section}
%\numberwithin{definition}{section} \numberwithin{remark}{section}

\newcommand{\R}{\mathbb{R}}
\newcommand{\ren}{\mathbb{R}^N}
\newcommand{\re}{\mathbb{R}}
\newcommand{\n}{\nabla}
\newcommand{\p}{\partial}
\newcommand{\iy}{\infty}
\newcommand{\pa}{\partial}
\newcommand{\fp}{\noindent}
\newcommand{\ms}{\medskip\vskip-.1cm}
\newcommand{\mpb}{\medskip}
\newcommand{\AAA}{{\bf A}}
\newcommand{\BB}{{\bf B}}
\newcommand{\CC}{{\bf C}}
\newcommand{\DD}{{\bf D}}
\newcommand{\EE}{{\bf E}}
\newcommand{\FF}{{\bf F}}
\newcommand{\GG}{{\bf G}}
\newcommand{\oo}{{\mathbf \omega}}
\newcommand{\Am}{{\bf A}_{2m}}
\newcommand{\CCC}{{\mathbf  C}}
\newcommand{\II}{{\mathrm{Im}}\,}
\newcommand{\RR}{{\mathrm{Re}}\,}
\newcommand{\eee}{{\mathrm  e}}
\newcommand{\LL}{L^2_\rho(\ren)}
\newcommand{\LLL}{L^2_{\rho^*}(\ren)}
\renewcommand{\a}{\alpha}
\newcommand{\g}{\gamma}
\newcommand{\G}{\Gamma}
\renewcommand{\d}{\delta}
\newcommand{\D}{\Delta}
\newcommand{\e}{\varepsilon}
\newcommand{\var}{\varphi}
\renewcommand{\l}{\lambda}
\renewcommand{\o}{\omega}
\renewcommand{\O}{\Omega}
\newcommand{\s}{\sigma}
\renewcommand{\t}{\tau}
\renewcommand{\th}{\theta}
\newcommand{\z}{\zeta}
\newcommand{\wx}{\widetilde x}
\newcommand{\wt}{\widetilde t}
\newcommand{\noi}{\noindent}
 %%%%%%%%%%%%%%%%%%%%%%%%%%%%%%%%%%%%%%%%%%%
\newcommand{\uu}{{\bf u}}
\newcommand{\xx}{{\bf x}}
\newcommand{\yy}{{\bf y}}
\newcommand{\zz}{{\bf z}}
\newcommand{\aaa}{{\bf a}}
\newcommand{\cc}{{\bf c}}
\newcommand{\jj}{{\bf j}}
\newcommand{\UU}{{\bf U}}
\newcommand{\YY}{{\bf Y}}
\newcommand{\HH}{{\bf H}}
\newcommand{\GGG}{{\bf G}}
\newcommand{\VV}{{\bf V}}
\newcommand{\ww}{{\bf w}}
\newcommand{\vv}{{\bf v}}
\newcommand{\hh}{{\bf h}}
\newcommand{\di}{{\rm div}\,}
\newcommand{\ii}{{\rm i}\,}
\newcommand\bR{{\mathbb R}}
\def\I{{\rm Id}}
%%%%%%%%%%%%%%%%%%%%%%%%%%%%%%%%%%
%%%%%%%%%%%%%%%%%%%%%%%%%%%%%%%%%%%%%   VAG, NEW
\newcommand{\inA}{\quad \mbox{in} \quad \ren \times \re_+}
\newcommand{\inB}{\quad \mbox{in} \quad}
\newcommand{\inC}{\quad \mbox{in} \quad \re \times \re_+}
\newcommand{\inD}{\quad \mbox{in} \quad \re}
\newcommand{\forA}{\quad \mbox{for} \quad}
\newcommand{\whereA}{,\quad \mbox{where} \quad}
\newcommand{\asA}{\quad \mbox{as} \quad}
\newcommand{\andA}{\quad \mbox{and} \quad}
\newcommand{\withA}{,\quad \mbox{with} \quad}
\newcommand{\orA}{,\quad \mbox{or} \quad}
\newcommand{\atA}{\quad \mbox{at} \quad}
\newcommand{\onA}{\quad \mbox{on} \quad}
\newcommand{\ef}{\eqref}
\newcommand{\mc}{\mathcal}
\newcommand{\mf}{\mathfrak}
\newcommand{\ssk}{\smallskip}
\newcommand{\LongA}{\quad \Longrightarrow \quad}
%%%%%%%%%%%%%%%%%%%%%%%%%%%%%%%%
%%%%%%%%%%%%%%%%%%%%%%%%%%%%%%%%%%
\def\com#1{\fbox{\parbox{6in}{\texttt{#1}}}}
%%%%%%%%%%%%%%%%%%%%%%%%%%%%%%%%%%
%%%%%%%%%%%%%%%%%%% From Paper1
\def\N{{\mathbb N}}
\def\A{{\cal A}}
\newcommand{\de}{\,d}
\newcommand{\eps}{\varepsilon}
\newcommand{\be}{\begin{equation}}
\newcommand{\ee}{\end{equation}}
\newcommand{\spt}{{\mbox spt}}
\newcommand{\ind}{{\mbox ind}}
\newcommand{\supp}{{\mbox supp}}
\newcommand{\dip}{\displaystyle}
\newcommand{\prt}{\partial}
\renewcommand{\theequation}{\thesection.\arabic{equation}}
\renewcommand{\baselinestretch}{1.1}
%%%%%%%%%%%%%%%%%%%%%%%%%%%%%%%%%%%%%%%%%%%%%%%
\newcommand{\Dm}{(-\D)^m}

%\title[linear parabolic problems]{Limiting problem of linear parabolic equations with spatio-temporal degenerated potentials}

%\title{Limit problem of linear parabolic equations with spatio-temporal degenerated potentials}

\title{Asymptotic limit of linear parabolic equations with spatio-temporal degenerated potentials}

\author{Pablo \`Alvarez-Caudevilla, Matthieu Bonnivard, Antoine Lemenant}

\date{\today}

%\setlength{\topmargin}{5mm} \pagestyle{myheadings}
%\markboth{}{Parabolic problems with spatio-temporal heterogeneities}

\maketitle

\begin{abstract} 
In this paper, we observe how the heat equation in a non-cylindrical domain can arise as the asymptotic limit of a parabolic problem in a cylindrical domain, by adding a potential that vanishes outside the limit domain. This can be seen as a parabolic version of a previous work by the first and last authors, concerning the stationary case~\cite{PaAn}. We provide a strong convergence result for the solution by use of energetic methods and $\Gamma$-convergence technics. Then, we establish an exponential decay estimate coming from an adaptation of an argument due to B. Simon.  \end{abstract}

\tableofcontents

%%%%%%%%%%%%%%%%%%%%%%%%%%%%%%%%%%%%%%%%%%%%%%%%%%%%%

%%%%%%%%%%%%%%%%%%%%%%%%%%%%%%%%%%%%%%%%%%%%%%%%%%%%%
\section{Introduction}
%%%%%%%%%%%%%%%%%%%%%%%%%%%%%%%%%%%%%%%%%%%%%%%%%%%%%%
For   $\Omega \subset \R^N$   open and  $T>0$, we define the cylinder $Q_T=\O\times (0,T)$. Let   $\lambda >0$ be a positive real parameter. For  $f_\lambda \in L^2(Q_T)$, $g_\l \in H_0^1(\Omega)$ and $a:Q_T \to \R^+$ a bounded measurable function, we consider the solution $u_\lambda$ of the parabolic problem
$$
(P_\lambda)
\left\{
\begin{array}{ll}
\partial_t u -\Delta u + \lambda a(x,t) u = f_\lambda & \hbox{in}\quad Q_T,\\
u =0 & \text{on}\quad \partial \Omega \times (0,T), \\
u(x,0)=g_\lambda(x) & \hbox{in}\quad \O.
\end{array}
\right.
$$
Since $(P_\l)$ is a classical parabolic problem, existence and regularity of solutions follow from standard theory well developed in the literature (see Section \ref{solP1}).  In particular, under our assumptions, $u\in L^2(0,T;H^1_0(\Omega))$ is continuous in time  with values in $L^2(\Omega)$ (thus the initial condition $u(x,0)=g_\lambda(x)$  is well defined   in $L^2(\O)$) and the equation is satisfied in a weak sense (see  Section \ref{solP1} for  an exact formulation).

In this paper we are interested in the limit  of $u_\lambda$ when $\lambda \to +\infty$. In particular, we 
 assume  spatial and temporal degeneracies for the potential a, which means that  
\be
\label{zeroset}
O_a:={\rm Int}\big(\{(x,t) \in Q_T \; : \; a(x,t)=0\}\big)\not = \emptyset.
\ee
We also assume that  $\partial O_a$ has zero Lebesgue measure.

In order to describe the results of this paper, let us start with  elementary observations.  Assume that, when $\lambda$ goes to $+\infty$, $f_\lambda$ converges  to $f$ and $g_\lambda$ converges to  
$g$, for instance  in $L^2$. Assume also that $u_\lambda$ converges weakly in 
$L^2(Q_T)$ to some $u\in L^{2}(Q_T)$.

Under those assumptions it is not very difficult  to get the following a priori bound using the equation in $(P_\lambda)$ (see  Lemma \ref{lem1})
\begin{eqnarray}
\lambda \int_{Q_T} a u^2_{\lambda} \;dxdt \leq C. \label{standardBound}
\end{eqnarray}
This shows that $u_\lambda$ converges strongly to $0$ in any set  of the form $\{ a(x,t)>\varepsilon\}$, for any $\varepsilon >0$.

Then, multiplying the equation in $(P_\lambda)$ by any   $\varphi \in C^{\infty}_0(O_a)$ we get, after some integration by parts (in this paper we shall denote $\nabla$ for $\nabla_x$, i.e. the gradient in space),
$$\int_{Q_T} u_\lambda \partial_t \varphi - \int_{Q_T}  u_\lambda\Delta \varphi  =\int_{Q_T} f_\lambda \varphi.$$ 
Passing to the limit, we obtain that  $\partial_t u -\Delta u   = f $ in $\mathcal{D}'(O_a)$.  Under  some suitable extra assumptions on the potential $a$, we will actually be able to prove that    the limit $u$ satisfies the following more precise problem:

 $$
(P_\infty)
\left\{
\begin{array}{l}
u \in L^2(0,T;H^1_0(\Omega)), \qquad u' \in L^2(0,T ; L^2(\Omega))  \\
u=0 \text{ a.e.  in }  Q_T \setminus O_a\\ 
\int_{Q_T} (u' v +\nabla u \nabla v )  = \int_{Q_T} {\color{blue} f} v ,\; \\
\hspace{2cm}   \text{ for all } v  \in L^2(0,T;H^1_0(\Omega)) \text{ s.t. } v=0 \text{ a.e. in }Q_T \setminus O_a\\
u(x,0)=g(x) \quad  \text{ in }  \O.
\end{array}
\right.
$$

%Furthermore, assuming some more conditions on $a$, we will   get some energetic bounds which implies that  $u$ must be a   solution for the following heat equation in the \emph{time dependent} domain $O_a\subset \Omega\times [0,T]$:
 
 Problem $(P_\infty)$, which arises here naturally as the limit problem associated with the family of problems $(P_\lambda)$, is a non standard Cauchy-Dirichlet problem for the heat equation since 
$O_a$ may, in general, not be cylindrical. This type of heat equation in a noncylindrical domain appears in many applications, and different approaches have been developed recently to solve problems related to $(P_\infty)$  (see for e.g.  \cite{boudin,Lieb,wang,novaga,DuPeng,savare} and the references therein). As a byproduct to our work, we have obtained an existence and uniqueness result for the problem $(P_\infty)$ (see Corollary \ref{existence}).

Furthermore, in this paper we study in more detail the  convergence of $u_\lambda$, when $\lambda$ goes to infinity.   
Our first result gives a sufficient condition on the potential $a$, for which the convergence of $u_\lambda$ to $u$ is stronger than a weak $L^2$ convergence. Indeed, assuming a monotonicity condition on the potential $a$, and using purely energetic and variational methods,  we 
obtain that the convergence holds strongly in $L^2(0,T;H^1(\Omega))$; see Section 5. Our approach can be seen as the continuation of a previous work \cite{PaAn}, where the stationary problem was studied using 
the theory of $\Gamma$-convergence,  as well as in \cite{PJ2} using a different analysis.

Here is our first main result.

\begin{theorem} \label{main1} For all $\lambda >0$, let $u_\lambda$ be the solution of $(P_\lambda)$ with $f_\lambda \in L^2(\Omega\times (0,T))$ and $g_\l \in H_0^1(\Omega)$. Assume that $a:\Omega \times [0,T]\to \R^+$ is a Lipschitz function which satisfies 
\begin{eqnarray}
\partial_t a(x,t) \leq 0\text{ a.e. in } Q_T. \label{monotonic}
\end{eqnarray}
Assume also that the initial condition $g_\lambda$ satisfies 
$$\sup_{\lambda>0}\left(\lambda \int_{\Omega} a(x,0)  g_\lambda(x)^2 dx\right)<+\infty,$$
converges weakly to $g$ in $L^2(\Omega)$, and that  $f_\l$ converges weakly to $f$ in $L^2(Q_T)$.

 Then  $u_\lambda$ converges strongly to $u$ in $L^2(0,T;H^1(\Omega))$, where $u$ is the unique  solution of $(P_\infty)$.
\end{theorem}

\begin{remark}\label{Increasing-in-time}
 In particular,  condition \eqref{monotonic} implies that the family of sets $\Omega_a(t)\subset\Omega$, defined for $t>0$ by $\Omega_a(t):=\{x\in \Omega, (x,t)\in O_a\}$, is increasing in time for the inclusion. In that case, by a slight abuse of terminology, we will often write simply that $O_a$ is increasing in time (for the inclusion).
 \end{remark}
 
 Our second result is a quantitative convergence of $u_\lambda$ to $0$, outside $O_a$ (in other words, away from the vanishing region), with very general assumptions on  $a$  (only continuous and  $O_a \not = \emptyset$), but in the special case when $f_\lambda=0$ in $Q_T \setminus O_a$. This   is obtained using an adaptation of an argument due to Simon   \cite{SimonII}, and proves that $u_\lambda$ decays exponentially fast 
 with respect to $\lambda$ to $0$ in the region $Q_T \setminus O_a$. Compared  to the standard bound \eqref{standardBound}, this results expresses that $u_\lambda$ goes to $0$ much faster than one could expect.  We also take the opportunity of this paper to write a similar estimate for the stationary problem (see Lemma~\ref{simon} in Section 6).

\begin{theorem} \label{main2}  For all $\lambda >0$, let $u_\lambda$ be the solution of $(P_\lambda)$ with $f_\lambda \in L^2(Q_T)$ and $g_\l \in H_0^1(\overline{O_a}\cap \{t=0\})$. Assume that $f_\lambda=0$ in $Q_T \setminus O_a$. Let $a:\overline{\Omega} \times [0,T]\to \R^+$ be a continuous function for which $O_a$ is non empty. For every $\varepsilon > 0$, define  $A_\varepsilon := \{(x,t) ;\, dist((x,t),O_a)>\varepsilon\}$. Then, for every $\varepsilon>0$, there exists a constant $C>0$ such that
$$
\sup_{\lambda >0}  \left(\l   e^{c_\varepsilon \sqrt{\l } }\int_{A_{\varepsilon}}   \, u_\l^2 \;dx\right) \leq  C,
$$
where  $c_\varepsilon:= \varepsilon \min_{(x,t) \in A_{\varepsilon/2} }a(x,t).$
\end{theorem}

The convergence of weak solutions of $(P_\lambda)$ was already observed in \cite{daners}  as a starting point for a more detailed analysis about the associated semigroup. This was then used   in \cite{daners} to analyse the asymptotic behaviour of a  non linear    periodic-parabolic problem of logistic type (firstly analysed by Hess \cite{Hess}) where the equation is the following, also   considered before in \cite{DuPeng},
\begin{eqnarray}
\partial_t u -\Delta u = \mu u-a(x,t) u^p,\label{logistic}
\end{eqnarray}
used in some models of population dynamics. A possible link between our Problem $(P_\lambda)$ and the non-linear equation \eqref{logistic} is coming from the fact that asymptotic limit of  the principal eigenvalue for the linear parabolic operator $\p_t-\D+\l a(x,t)$ plays a role in the dynamical behaviour of non-linear  logistic equation (cf. \cite{A-LG, daners,DuPeng,GGLS}).  We thus believe that the results and  techniques developed  
in the present paper could possibly  be used in the study of more general equations such as \eqref{logistic}.

Furthermore, another possible application of our results could be for numerical purposes. Indeed, for the ones who would be interested by computing a numerical solution of the non-cylindrical limiting problem $(P_\infty)$, one could use the   cylindrical problem $(P_\lambda)$ for a large $\lambda$, much easier to compute via standard methods. The strong convergence stated in Theorem \ref{main1}, together with the exponential rate of convergence stated in Theorem \ref{main2}, give some good estimates about the difference between those two different solutions.

%Theorem~\ref{main2}  and Theorem~\ref{main1} are then used in a last section, where  we focus on the limiting problem $(P_\infty)$  from a numerical point of view. We use the approximation by $(P_\lambda)$ justified by Theorem \ref{main1} and Theorem \ref{main2}, as  a method to obtain approximated solutions. Indeed, Problem $(P_\infty)$ may not be easily implemented by use of standard numerical analysis, since the domain $O_a$ may not be cylindrical. On the other hand $(P_\lambda)$ is a standard parabolic problem that can by computed very efficiently. Therefore, to compute an approximated solution of $(P_\infty)$, for an arbitrary domain $O_a$, we consider a solution of $(P_\lambda)$ with a large $\lambda$ with $a(x,t):=dist((x,t), O_a)$ and we take $f{\bf 1}_{O_a}$  as second member so as to force $u$ being almost $0$ outside $O_a$, as stipulated by Theorem \ref{main2}. 

%%%%%%%%%%%%%%%%%%%%%%%%%%%%%%%%%%%%%%%%%%%%%%%

\section{The stationary problem}

%%%%%%%%%%%%%%%%%%%%%%%%%%%%%%%%%%%%%%%%%%%%%%%%

This section concerns only the stationary problem. In particular, throughout the section, all functions $u, a, f$, \emph{etc.}, will be functions of $x \in \Omega$ (and independent of $t$). 

We assume $\Omega \subset \mathbb{R}^N$ to  
be an open and bounded domain and  $a:\overline{\Omega} \to \mathbb{R}^+$ be a measurable and bounded  non-negative function. We suppose that   
\begin{eqnarray}
K_a:=\{x \in \overline{\Omega} ; a(x)=0 \}\subset \Omega \text{ is a closed set in } \mathbb{R}^N. \label{nicepot}
\end{eqnarray}
Moreover, we assume that 
\begin{eqnarray}
\O_a:={\rm Int}(K_a)\not = \emptyset.  \label{nontriviality}
\end{eqnarray}
Under hypothesis \eqref{nicepot} we know that 
$$H^1_0(K_a):=H^1(\R^N)\cap\{u=0 \text{ q.e. in } \R^N\setminus K_a\}=H^1(\R^N)\cap\{u=0 \text{ a.e. in } \R^N\setminus K_a\},$$
and hypothesis \eqref{nontriviality} implies that 
$$H^1_0(K_a) \not = \{0\}.$$

Notice that we are working with a functional space of the form $H^1_0(A)$, where $A$ is a closed subset of $\R^N$. Therefore, we do not claim that
$H^1_0(A)=H^1_0(\mathrm{Int}(A))$,  which is true only under more regularity assumptions on the set $A$.

Furthermore, we define the functionals  $E_\lambda$ and $E$ on $L^2(\Omega)$ as follows. 
\be\label{Def:Elambda}
E_\lambda(u)=
\left\{
\begin{array}{ll}
\int_{\Omega} |\nabla u|^2 + \lambda a u^2  \; dx  & \text{ if } u \in H^1_0(\Omega)\\
+\infty  &\text{ otherwise.}
\end{array}
\right.
\ee
$$
E(u)=
\left\{
\begin{array}{ll}
\int_{\Omega} |\nabla u|^2  \; dx & \text{ if } u \in H^1_0(K_a)\\
+\infty  &\text{ otherwise.}
\end{array}
\right.
$$
The following result was already stated and used in \cite{PaAn}. For the sake of completeness, we reproduce the proof here and refer the reader to   \cite{PaAn} for the connection of this result with $\G$-convergence and several examples.  

\begin{proposition}\label{thecor} Let  $f_\lambda \in L^2(\Omega)$ be a family of functions indexed by some real parameter $\lambda>0$ and 
 uniformly bounded in $L^2(\O)$. Moreover, 
assume that $f_\lambda$ converges to a function $f \in L^2(\Omega)$ in the weak topology of $L^2(\Omega)$, when $\lambda$ tends to $+\infty$. Then the unique solution $u_\lambda$ of the problem 
$$
(P^{s}_{\l})
\left\{
\begin{array}{l}
-\Delta  u_\lambda + \lambda a u_\lambda = f_\lambda  \\
\hspace{0.2cm} u_\lambda \in H^1_0(\Omega),
\end{array}
\right.
$$
converges strongly in $H^1(\Omega)$, when $\lambda \to +\infty$, to the unique solution of the problem
$$
(P^s_\infty)
\left\{
\begin{array}{l}
-\Delta u = f \\
\hspace{0.2cm} u \in H^1_0(K_a).
\end{array}
\right.
$$
\end{proposition}
\begin{proof} This is a standard consequence of the $\Gamma$-convergence of energies $E_\l$, which relies on the fact that $u_\l$ is the unique  minimizer in $H^1_0(\Omega)$ for 
$$v\mapsto E_\l(v)-2\int_{\Omega} f_\l v,$$
whereas $u$ is the unique minimizer in $H^1_0(K_a)$ for 
$$v\mapsto E(v)-2\int_{\Omega} fv.$$

Let us write the full details of the proof. For any $\lambda>0$, let $u_\lambda$ be the   solution of $(P^{s}_\l)$. We first prove that $\{u_\lambda\}_{\lambda >0}$ is  compact in $L^2(\Omega)$. This comes from the energy equality 
$$\int_{\Omega} (|\nabla u_\lambda|^2 + \lambda a u_\lambda^2) \; dx= \int_{\Omega} f_\lambda u_\lambda \; dx,$$
which implies
$$\int_{\Omega} |\nabla u_\lambda|^2 \leq \| f_\lambda \|_{L^2(\O)} \|u_\lambda \|_{L^2(\O)}\leq C \|u_\lambda \|_{L^2(\O)}\quad \hbox{with $C$ a positive constant}.$$
Thanks to Poincar\'e's inequality we also have that
$$\|u_\lambda \|_{L^2(\O)}^2 \leq C(\Omega)\int_{\Omega} |\nabla u_\lambda|^2\; dx,$$
which finally proves that  $u_\lambda$ is  uniformly bounded in $H^1_0(\Omega)$.

Now let  $w$ be any point in the $L^2$-adherence of the family $\{u_\lambda\}_{\lambda >0}$. In other words, there exists a subsequence, still denoted by 
$u_\lambda$, converging strongly in $L^2$ to $w$. Since $u_\lambda$ 
is bounded in $H^1(\O)$, we can assume, up to extracting a further subsequence, that $u_\lambda$ converges weakly in $H^1(\O)$ to a function that must necessarily be $w$.
Now let $u$ be the solution of the limit problem $(P^s_\infty)$. Since $u_\lambda$ is a minimizer of 
\begin{equation}
\label{qu23}
u\mapsto E_\lambda(u)-2\int_{\Omega } f_\lambda u \; dx.
\end{equation}
Denoting 
$$\alpha(\lambda)=\min_{v\in H_0^1(\Omega)} E_\lambda(u)-2\int_{\Omega } f_\lambda u,$$
for any $\lambda, \lambda'>0$ and such that $\lambda <\lambda'$, we find that
$$\alpha(\lambda)\leq E_\lambda(u_{\lambda'})=\alpha(\lambda')+(\lambda-\lambda')\int_\Omega a u_{\lambda'}^2 <\alpha(\lambda'),$$
for every $u_{\lambda'}\in  H_0^1(\Omega)$. Hence, $\l \to \alpha(\l)$ is increasing and by continuity of $\l \to u_\l$ we deduce that
$$\alpha'(\l)= \lim_{\l' \to \l} \frac{|\alpha(\l)-\alpha(\l')|}{|\l-\l'|}=\int_{\Omega}a v_{\l}^2$$
and $\alpha'(\l)$ is continuous so that $\alpha \in C^1$. Now suppose that
\begin{eqnarray}
\sup_{\l >0} \alpha(\l)<+\infty, \label{stop}
\end{eqnarray}
and assume by contradiction that there exists $\lambda_0,\varepsilon>0$ such that
$$ \l \int_{\Omega} a u_\l^2 =\l \alpha'(\l) > \varepsilon >0 \quad \quad \text{ for } \l > \l_0.$$
Thus, integrating the last inequality between $\l_0$ and $\l$ we deduce that
$$\alpha(\l)\geq \alpha(\l_0)+\varepsilon \ln \left(\frac{\l}{\l_0}\right)$$
which contradicts \eqref{stop}. Consequently, since $a  u=0$ we have 
$$E_\lambda(u_\lambda)-2\int_{\Omega } f_\lambda u_\lambda \; dx \leq E_\lambda(  u)-2\int_{\Omega } f_\lambda   u \; dx=E( u)-2\int_{\Omega } f_\lambda   u \; dx.$$
Hence, letting $\lambda$ go to infinity in the previous inequality, it follows that  
\begin{eqnarray}
 E(w)-2\int_{\Omega} fw  \; dx &\leq &\liminf_{\lambda }\left(E_\lambda(u_{\lambda})-2\int_{\Omega } f_{\lambda} u_\lambda  \; dx \right) \notag \\
 &\leq &\limsup_{\lambda }\left(E_\lambda(u_{\lambda})-2\int_{\Omega } f_{\lambda} u_\lambda  \; dx \right) \notag \\
 &\leq&    E(  u)-2\int_{\Omega } f  u  \; dx, \label{limsup}
\end{eqnarray}
which shows that $w$ is a minimizer, and thus $w=  u$. By uniqueness of the adherence point, we infer that the whole sequence $u_\lambda$ converges strongly in $L^2$ to $u$ (and weakly in $H^1$).

It remains to prove the strong convergence in $H^1$. To do so, it is enough to prove 
$$\|\nabla u_\lambda\|_{L^2(\O)} \to \|\nabla u\|_{L^2(\O)}.$$   
Due to the weak convergence in $H^1(\O)$ (up to subsequences) we already have
$$ \|\nabla u\|_{L^2(\O)}\leq \liminf_{\lambda}\|\nabla u_\lambda\|_{L^2(\O)}, $$
and going  back  to \eqref{limsup}  we get the reverse inequality, with a limsup.

The proof of convergence of the whole sequence follows by uniqueness of the adherent point in $H^1(\O)$.
\end{proof}

\begin{remark}  Notice that when $u$ is a solution of $(P^s_\infty)$, then $-\Delta u=f$ only in $\mathrm{Int}(K_a)$ and $-\Delta u = 0$ in $K_a^c$. However,  in general $-\Delta u$ has a singular part   on $\partial K_a$. Typically, if $K_a$ is for instance a set of finite perimeter, then in the distributional sense in $\Omega$, 
$$-\Delta u = f {\bf 1}_{K_a} + \frac{\partial u}{\partial \nu} \mathcal{H}^{N-1}|_{\partial K_a},$$
where $\nu$ is the outer normal on $\partial K_a $ and $\mathcal{H}^{N-1}$ is the $N-1$ dimensional Hausdorff measure.
\end{remark}

As a consequence of Proposition \ref{thecor}, we easily obtain the following result.

\begin{proposition} Assume that $f_\lambda$ converges weakly to a function $f$ in $L^2(\O)$. For any $\lambda>0$, let $u_\lambda$ be the solution of Problem $(P^s_\l)$. Then, when $\lambda\rightarrow \infty$,
\begin{eqnarray}
\lambda \int_{\Omega} a u_\lambda^2 \; dx \to 0, \label{niceconv1}
\end{eqnarray}
\begin{eqnarray}
\lambda a u_\lambda \to f{\bf 1}_{\Omega \setminus K_a} +(\Delta u)|_{\partial K_a} \text{ in } \mathcal{D}'(\Omega), \label{niceconv2}
\end{eqnarray}
where $u$ is the solution of $(P^s_\infty)$. Moreover, the convergence in \eqref{niceconv2} holds in the weak-$*$ topology of $H^{-1}$.
\end{proposition}

\begin{proof}  Due to Proposition \ref{thecor} we know that $u_\lambda$ converges strongly in $H^1(\Omega)$ to $u$, the solution of Problem $(P^s_\infty)$. In particular, from the fact that 
$$\int_{\Omega}|\nabla u_\lambda |^2 \; dx \to \int_{\Omega }|\nabla u|^2 \; dx=\int_{\Omega} uf\; dx,$$
and
$$\int_{\Omega} u_\lambda f_\lambda \; dx \to \int_{\Omega} uf  \; dx,$$
  passing to the limit in the following energy equality  
\begin{eqnarray}
\int_{\Omega} |\nabla u_\lambda|^2 \; dx + \lambda \int_{\Omega} a u_\l^2 \; dx = \int_{\Omega} u_\lambda f_\lambda \; dx, \label{firstCo}
\end{eqnarray}
we obtain \eqref{niceconv1}. 
Next, let us now prove \eqref{niceconv2}. Thus, since $u_\lambda$ is a solution of $(P^s_\l)$ then, for every test function $\psi \in C^\infty_c(\Omega)$, after integrating by parts in $\O$ we arrive at
\begin{equation}
\label{newrefr}
\int_{\Omega} u_\lambda (-\Delta \psi) \; dx +\l \int_{\Omega} a u_\lambda \psi \; dx = \int_{\Omega}  f_\lambda \psi\; dx.
\end{equation}
Passing to the limit we obtain that  $\lambda a u_\l \to f+\Delta u$ in $\mathcal{D}'(\Omega)$. 
Now returning to \eqref{newrefr}, we can  write, for every $\psi$ satisfying $\| \psi \|_{H^1(\Omega)} \leq 1$,
$$ \left|\l \int_{\Omega} a u_\lambda \psi \; dx \right| \leq  \|f_\l\|_2 + \|\nabla u_\l\|_2 \leq C.$$
Taking the supremum in $\psi$ we get 
$$\|\lambda a u_\l\|_{H^{-1}}\leq C.$$
Therefore, $\lambda a u_\l$ is weakly-$*$ sequentially compact in $H^{-1}$ and we obtain the convergence by uniqueness of the limit in the distributional sense.
\end{proof}

\section{Existence and regularity  of solutions for $(P_\l)$}
\label{solP1}

In order to define properly a solution for $(P_\l)$, we first recall the definition of the spaces $L^p(0,T;X)$, with $X$ a Banach space, which consist of all (strongly) measurable functions (see \cite[Appendix E.5]{Ev}) $ u:[0,T]\to X$ such that  
$$\| u \|_{L^p(0,T;X)} =\left(\int_0^T \| u(t)\|_X^p dt\right)^{1/p}< +\infty,$$
for $1\leq p <+\infty$, and 
$$\| u \|_{L^\infty(0,T;X)} = \underset{t\in (0,T)}{\rm ess\;sup}  \| u(t)\|_X< +\infty.$$
  For simplicity we will sometimes  use the following notation for $p=2$ and $X=L^2(\Omega)$ :
$$\|\cdot\|_2\equiv \|\cdot\|_{L^2(0,T;L^2(\O))}.$$
We will also use the notation $u(x,t)=u(t)(x)$ for $(x,t)\in \Omega \times (0,T)$.

Next, we will denote by $u'$ the derivative of $u$ in the $t$ variable, intended in the following  weak sense: we say that $u'=v$, with 
$u, v  \in L^2(0,T;X)$ and 
$$\int_0^T \varphi'(t)u(t) dt = -\int_{0}^T\varphi(t)v(t) dt $$
for all scalar test functions $\varphi \in C^\infty_0(0,T)$.  The space $H^1(0,T;X)$ consists of all functions $u\in L^2(0,T;X)$ such that $u' \in L^2(0,T;X)$.

We will often use the following remark.

\begin{remark}   \label{abscont} By \cite[Theorem 3 page 303]{Ev}, if $u\in L^2(0,T;H^1_0(\Omega))$ and $u' \in L^2(0,T;H^{-1}(\Omega))$, then $u \in C([0,T],L^2(\Omega)$ (after possibly being redefined on a set of measure zero). Moreover the mapping $t\mapsto \|u(t)\|^2_{L^2(\Omega)}$ is absolutely continuous and  
$$\frac{d}{dt}\|u(t)\|^2_{L^2(\Omega)}=2\langle u'(t),u(t)\rangle_{L^2(\Omega)}.$$
\end{remark}

%%%%%%%%%%%%%%%%%%%%%%%%%%%%%%%%%%%

In this section we collect some useful information about the solution $u_\l$ of problem $(P_\lambda)$ coming from the classical theory of parabolic problems that can be directly found in the literature.

Firstly, existence and uniqueness of a weak solution  $u_\lambda$ for the problem 
$(P_\lambda)$ follows from  the standard Galerkin method;  see \cite[Theorems 3 and 4, Section 7.1]{Ev} for further details. 
According to this theory, a weak solution means that:
$$
(P_\l)
\left\{
\begin{array}{l}
u \in L^2(0,T;H^1_0(\Omega)), \qquad u' \in L^2(0,T ; H^{-1}(\Omega))  \\
 \int_0^T \left\langle u'(t),v(t) \right\rangle_{(H^{-1},H_0^1)}+  \int_{Q_T} (\nabla u\cdot  \nabla v   + \lambda  a\, u\,v)  = \int_{Q_T} f_\l v \\
\hspace{2cm}\text{ for all  }  v \in L^2(0,T;H^1_0(\Omega)), \\ 
 u(0)=g_\l(x)\qquad \hbox{in}\quad L^(\Omega).
\end{array}
\right.
$$

Remember that  by Remark \ref{abscont} above, such weak solution $u$  is continuous in time with values in $L^2(\Omega)$ so that the initial condition  
makes sense. In the rest of the paper, $(P_\l)$ will always refer to the above precise formulation of the problem that was first stated in the Introduction.

Next, thanks to \cite[Theorem 5, Section 7.1]{Ev}, by considering $\l a u$  as a right hand side term (in $L^2(\Omega\times (0,T))$), we have the following.
\begin{lemma}\label{lemaB}
Let $\lambda>0$,  $g_\l \in H^1_0(\Omega)$,   $f_\l\in L^2(0,T;L^2(\O))$, and let $u_\l$ be the weak solution to $(P_\l)$. Then, 
$$u_\l \in L^2(0,T;H^2(\Omega)) \cap L^\infty(0,T;H^1_0(\Omega)), \quad u_\l' \in L^2(0,T;L^2(\Omega)),$$
and $u_\l$ satisfies the following estimate:
\be
\label{bound1}
\begin{split}
   \sup_{0\leq t \leq T} \|u_\l(t)\|_{H^1_0(\Omega)} & +\|u_\l\|_{L^2(0,T;H^2(\Omega))}+\|u_\l'\|_{2}
  \\ &  \leq C\left(\lambda \|au_\l\|_{2}+\|f_\l\|_{2}+\|g_\l\|_{H^1_0(\Omega)}\right), 
\end{split}
\ee
where the constant $C$ depends only on $\O$ and $T$.
\end{lemma}

\begin{remark}
Notice that the bound \eqref{bound1} is not very useful when $\lambda \to +\infty$ since what we usually control is $\sqrt{\lambda}\|au\|_{2}$ (shown below in Lemma \ref{lem1}) 
but not 
$\lambda \|au\|_{2}$. Thus the right hand side blows-up a priori.
\end{remark}

%%%%%%%%%%%%%%%%%%%%%%%%%%%%%%%%
\section{Uniqueness of solution for $(P_\infty)$}
%%%%%%%%%%%%%%%%%%%%%%%%%%%%%%%%

In this section we focus on the following problem that will arise as the limit of   $(P_\l)$. Our notion of solution for the problem $\partial_t u -\Delta u   = f$ in $O_a$ will precisely  be the following :

 $$
(P_\infty)
\left\{
\begin{array}{l}
u \in L^2(0,T;H^1_0(\Omega)), \qquad u' \in L^2(0,T ; L^2(\Omega))  \\
u=0 \text{ a.e.  in }  Q_T \setminus O_a\\ 
\int_{Q_T} (u' v +\nabla u\cdot \nabla v)  = \int_{Q_T} f v ,\; \\
\hspace{2cm}   \text{ for all } v  \in L^2(0,T;H^1_0(\Omega)) \text{ s.t. } v=0 \text{ a.e. in }Q_T \setminus O_a\\
u(0)=g(x) \quad  \text{ in }  \O.
\end{array}
\right.
$$

As a byproduct of  Section \ref{solP2}  we will prove the existence of a solution for the problem $(P_\infty)$, as a limit of solutions for $(P_\l)$. In this section, we prove the uniqueness which follows from  a simple energy bound. Notice  that a solution $u$ to $(P_\infty)$ is continuous in time (see Remark \ref{abscont}) thus the initial condition $u(x,0)=g(x)$ makes sense  in $L^2(\Omega)$.

\begin{proposition} \label{unique}Any solution $u$ of $(P_\infty)$ satisfies the following energy bound
\begin{eqnarray}
\frac{1}{4}\sup_{t\in (0,T)}\|u(t)\|_{L^2(\Omega)}^2  +  \| \nabla  u \|_2^2  \leq \frac{1}{2}\|g\|_{L^2(\Omega)}^2+   T \|f\|_2. \label{energyInf}
\end{eqnarray}
Consequently, there exists at most one solution to problem $(P_\infty)$.
\end{proposition}
\begin{proof} Let $u$ be a solution to $(P_\infty)$, and $s\in (0,T)$. Choosing $v=u\,\mathbf{1}_{(0,s)}$ (where $\mathbf{1}_{(0,s)}$ is the characteristic function of $(0,s)$) in the weak formulation of the problem, we deduce that   
\begin{eqnarray}
\int_{0}^{s} \int_{\Omega} u' u  \;dxdt +  \int_{0}^{s}\int_{\Omega} |\nabla u|^2 \; dx dt = \int_{0}^{s} \int_{\Omega} f u \;dx dt. \label{estimsoir2}
\end{eqnarray}
Now applying Remark \ref{abscont} and using the fact that  $u \in L^2(0,T;H^1_0(\Omega))$ and $u' \in L^2(0,T ; L^2(\Omega))$ we obtain  that  $t\mapsto \|u(t)\|^2_2$ is absolutely continuous, and for a.e. $t$, there holds
$$
\frac{d}{dt}\|u(t)\|^2_{L^2(\Omega)}=2\langle u'(t),u(t)\rangle_{L^2(\Omega)}. 
$$
Returning to \eqref{estimsoir2} we get
\begin{eqnarray}
\frac{1}{2}\|u(s)\|_{L^2(\Omega)}^2 -\frac{1}{2}\|u(0)\|_{L^2(\Omega)}^2 +  \int_{0}^{s}\int_{\Omega} |\nabla u|^2 \; dx dt = \int_{0}^{s} \int_{\Omega} f u \;dx dt. \label{estimsoir22}
\end{eqnarray}
By Young's inequality we have 
\begin{eqnarray}
 \left|\int_{0}^{s} \int_{\Omega} f u \;dx dt\right| & \leq& \frac{\alpha}{2}\|f\|_{L^2(\Omega \times (0,s))}^2 +\frac{1}{2\alpha}\|u\|_{L^2(\Omega\times (0,s))}^2 \notag\\
& \leq&   \frac{\alpha}{2}\|f\|_{L^2(\Omega \times (0,s))}^2 + \frac{T}{2\alpha}\sup_{t\in (0,T)}\|u(t)\|_{L^2(\Omega)}^2. \label{estimsoir3}
\end{eqnarray}
Setting $\alpha=2T$, estimating \eqref{estimsoir22} by \eqref{estimsoir3} and  passing to the supremum in $s\in (0,T)$ finally gives
\begin{eqnarray}
\frac{1}{4}\sup_{t\in (0,T)}\|u(t)\|_{L^2(\Omega)}^2  +  \| \nabla  u \|_2^2  \leq \frac{1}{2}\|g\|_{L^2(\Omega)}^2+   T \|f\|_2^2,\notag
\end{eqnarray}
as desired.

Now assume that $u_1$ and $u_2$ are two solutions of $(P_\infty)$, and set $w:=u_1-u_2$. Then $w$ is a solution of $(P_\infty)$ with $f=0$ and $g=0$. 
Therefore,  applying \eqref{energyInf} to $w$ automatically gives $w=0$, which proves the uniqueness of the solution of $(P_\infty)$.
\end{proof}

%Here we finally state all we know about the limit function $u$.

%\begin{itemize}
%\item \item $u =0$ a.e. in $\O_a^c$. In particular, $u \in H^1_0(\Omega_a(t))$ for a.e. $t$.
%\item $u$ satisfies $u'-\Delta u = f$ in $\mathcal{D}'(O_a)$, hence, by elliptic regularity, also in the $H^2(O_a)$ sense.
%\item we deduce that for a.e. $t$, $u$ is the unique solution in $H^1_0(O_a(t))$ of $-\Delta u =f-u'$.  
%This means that the $L^2(\Omega \times [0,T])$ function $u'$ is equal to $\Delta u + f$ in $O_a$ and equal to $0$ in $O_a^c$. Since $u=0$ on $O_a^c$, we can therefore re-write \eqref{joyce} as
 %\begin{eqnarray}
%\frac{d}{dt}\|u(t)\|^2_{L^2(\Omega)}=2\langle f+\Delta u,u(t)\rangle.  \label{joyce2}
%\end{eqnarray}
%Notice that we do not have $u'=f+\Delta u$ globally since $\Delta u$ may not be an $L^2$ function in the whole $\Omega$ (has some singular part on $\partial O_a$). Nevertheless, \eqref{joyce2} holds true.
%\end{itemize}

\section{Convergence of $u_\l$}
\label{solP2}
We now analyse  the convergence of $u_\l$, which will follow from energy bounds for $u_\l$ and $u_\l'$. As already mentioned before, the standard energy bound for the solutions of $(P_\l)$ that is stated in Lemma \ref{lemaB}, blows up a priori when $\lambda$ goes to $+\infty$.  Our goal in the sequel is   to get better estimates, uniform in $\lambda$. The price to pay is the condition  $\partial_t a \leq 0$ which implies that $O_a$ is nondecreasing in time (for the set inclusion).

%%%%%%%%%%%%%%%%%%%%%%%%%%%%%%%%%%%%%%%%%%%%%%%%%%%%%%%%%%
\subsection{First energy bound} 

\begin{lemma}
\label{lem1}
Assume that $g_\l \in L^2(\Omega)$ and $f_\l\in L^2(0,T;L^2(\O))$, and let $u_\l$ be the weak solution of problem $(P_\l)$.  Then,
\begin{eqnarray}
 \frac{1}{4}\sup_{t \in (0,T)}\|u_\l(t)\|^2_{L^2(\Omega)}+ \| \nabla u_\l\|_{2}^2 + \lambda \int_0^T\int_{\Omega} a u_\l^2\, dx dt  \leq  \|g_\l\|_{L^2(\O)}^2+T\|f_\l\|_2^2.  
 \label{bound2}
 \end{eqnarray}
\end{lemma}
\begin{proof}
Let $u_\l$ be a solution of $(P_\l)$ and $s\in (0,T)$. Testing with $v=u\,\mathbf{1}_{[0,s]}$ in the weak formulation of $(P_\l)$
\begin{align*}
\frac{1}{2}\|u_\l(s)\|_{L^2(\Omega)}^2 -\frac{1}{2}\|u_\l(0)\|_{L^2(\Omega)}^2 +  \int_{0}^{s}\int_{\Omega} |\nabla u_\l|^2 \; dx dt + \lambda\int_0^{s}\int_{\O}au_\l^2\, dxdt \nonumber \\
= \int_{0}^{s} \int_{\Omega} f_\l u_\l \;dx dt. 
\end{align*}
and arguing as in the proof of Proposition \ref{unique}, we obtain \eqref{bound2}, so that we omit the details.
 \end{proof}
 
%\vspace{0.5cm}

%%%%%%%%%%%%%%%%%%%%%%%%%%%%%%%%%%%%%%%%%%%%%%%%%%%%%%%%%%%%%%%%%%
\subsection{Second energy bound} We now derive a uniform bound on $\|u_\l'\|_2$. To this end, we will assume a time-monotonicity condition on $a$.

%\begin{definition}[Assumptions (A)] We will say that we are under assumptions (A) if the following holds:  $a:Q_T\to \R^+$ is Lipschitz and   
%\begin{eqnarray}
%\partial_t a(x,t)\leq 0 \text{ for a.e. }(x,t)\in Q_T. \label{Assump1}
%\end{eqnarray}
%Moreover   
%$$ \quad \sup_{\lambda}\lambda \int_{\Omega} a(0) g_\lambda^2\leq C.$$
%Finally if $f_\l   \in L^2(0,T ; L^2(\Omega))$ and  $g_\l \in H^1_0(\Omega)$, $f_\l$ converges to $f$ weakly in $L^2(Q_T)$ and $g_\l$ converges weakly in $L^2(\O)$ to $g$. \end{definition}

\begin{definition}[Assumption (A)] Assume that $g_\l \in L^2(\Omega)$ and $f_\l\in L^2(0,T;L^2(\O))$. We say that  Assumption (A) hold if  $a:Q_T\to \R^+$ is Lipschitz and   
\begin{eqnarray}
\partial_t a(x,t)\leq 0 \text{ for a.e. }(x,t)\in Q_T. \label{Assump1}
\end{eqnarray}
 \end{definition}

\begin{lemma}
\label{lem2} We suppose that Assumption (A) holds.  Then, the solution  $u_\l$  of $(P_\l)$ satisfies the estimate:
\begin{align}
\label{derbound}
&\int_0^T\int_\O (u_\l')^2\, dxdt + \sup_{s\in (0,T)}\left(\int_\O |\nabla u_\l(s)|^2\, dx\right) \nonumber \\
& \leq \int_0^T\int_\O f_\l^2\, dxdt+\int_{\Omega}|\nabla u_\l(0)|^2\, dx+\lambda \int_{\Omega} a(0) g_\l^2\, dx.
\end{align}
\end{lemma}
\begin{proof}
 Thanks to Lemma \ref{lemaB}, we know that  $u_\l' \in L^2(0,T ; H^1_0(\Omega))$. Consequently, for every $s\in (0,T)$, the function $v:=u_\l'\, \mathbf{1}_{(0,s)}$ is an admissible test function in the weak formulation of $(P_\l)$. Hence, we obtain the identity
$$\int_0^s\int_{\Omega}(u_\l')^2\, dxdt+\int_0^s\int_{\Omega}\nabla u_\l \cdot \nabla u_\l'\, dxdt+\lambda \int_0^s\int_{\Omega} a u_\l u_\l' \, dx dt= \int_0^s\int_{\Omega}f_\l u_\l'\, dx dt,$$
or written differently (applying Remark \ref{abscont}),
\begin{align*}
\int_0^s\int_{\Omega}(u_\l')^2\, dx dt+\int_0^s\left(\frac{1}{2}\int_{\Omega}|\nabla u_\l|^2\, dx\right)'\, dt+\lambda\int_0^s  \left[\left(\frac{1}{2}\int_{\Omega} a u_\l^2\, dx\right)'-  \frac{1}{2}\int_{\Omega}a'u_\l^2\, dx\right]dt\\
 = \int_0^s\int_{\Omega}f_\l u_\l'\, dx dt.
\end{align*}
This yields
\begin{eqnarray*}
\int_0^s\int_{\Omega}(u_\l')^2\, dx dt+\frac{1}{2}\int_{\Omega}|\nabla u_\l(s)|^2\, dx+ \frac{\lambda }{2}\int_{\Omega} a(s) u_\l(s)^2\, dx- \frac{\lambda}{2}\int_0^s\int_{\Omega}a' u_\l^2\, dxdt
\\
= \int_0^s \int_{\Omega}f_\l u_\l'\, dx dt+\frac{1}{2}\int_{\Omega}|\nabla u_\l(0)|^2\, dx +  \frac{\lambda }{2}\int_{\Omega} a(0) u_\l(0)^2\, dx.
\end{eqnarray*}
By Young's inequality, 
$$\int_0^s\int_{\Omega}f_\l u_\l'\, dxdt \leq \frac{1}{2}\int_0^s\int_\O f_\l^2\, dxdt+\frac{1}{2}\int_0^s\int_\O (u_\l')^2\, dxdt,$$
so that we obtain
\begin{eqnarray*}
\int_0^s\int_{\Omega}(u_\l')^2\, dx dt+ \int_{\Omega}|\nabla u_\l(s)|^2\, dx+ \lambda\int_{\Omega} a(s) u_\l(s)^2\, dx- \lambda\int_0^s\int_{\Omega}a'u_\l^2\, dxdt
\\
\leq \int_0^s\int_\O|f_\l|^2\, dxdt+ \int_{\Omega}|\nabla u_\l(0)|^2\, dx +   \lambda\int_{\Omega} a(0) u_\l(0)^2\, dx.
\end{eqnarray*}
Finally, using Assumption (A), the initial condition on $u_\l(0)$ and passing to the supremum in $s$, we conclude that estimate \eqref{derbound} holds.
\end{proof}

\subsection{Weak convergence of  solutions} 
\noindent Using the previous energy estimates, we first establish the weak convergence of $u_\l$ to the solution $u$ of Problem $(P_\infty)$, under  \emph{Assumption (A)}, and supposing certain bounds on the right hand side $f_\l$ and on the initial data $g_\l$.

\begin{proposition}\label{weakConv} 
Assume that $a$ satisfies \emph{Assumption (A)}.
Let $\{f_\lambda\}$ be a bounded sequence in  $L^2(Q_T)$ and $\{g_\lambda\}$ be a bounded sequence in $H^1_0(\O)$, satisfying
\be\label{Bound:glambda}
\sup_{\lambda}\left( \lambda\int_\O a(0) g_\l^2\, dx \right) < \infty. 
\ee
Up to extracting subsequences, we can assume that $f_\l$ converges weakly to a function $f$ in $L^2(Q_T)$, and $g_\l$ converges weakly to a function $g\in H^1_0(\O)$. 

Let $u_\l$ be the solution of $(P_\l)$. Then $u_\l$ converges weakly in $L^2(Q_T)$ to the unique solution $u$ of problem $(P_\infty)$.
\end{proposition}

\begin{proof} We know by Lemma \ref{lem1} that $u_\l$ is uniformly bounded in $L^2(0,T;H^1_0(\Omega))$, thus converges weakly (up to extracting  a subsequence) in $L^2(0,T;H^1_0(\Omega))$ to some function $u \in L^2(0,T;H^1_0(\Omega))$.  Under \emph{Assumption (A)}, we also know, thanks to Lemma \ref{lem2}, that 
$$\|u'_\l\|_{L^2(Q_T)}\leq C,$$
so that, $u'_\l$ converges also weakly in $L^2(Q_T)$ (up to extracting a further subsequence) to some limit $w \in L^2(Q_T)$, which must be equal to $u'$ by uniqueness of the limit in $\mathcal{D}'(Q_T)$. This shows that $u'\in L^2(Q_T)$.

Next, due to \eqref{bound2} we know that 
$$\sup_{\lambda} \left(\lambda \int_{0}^T\int_{\O} a u_\l^2\, dxdt\right) \leq C,$$
which implies that, at the limit, $u$ must be equal to zero a.e. on any set of the form $\{a>\varepsilon\}$, with $\varepsilon>0$. By considering the union for $n\in \mathbb{N}^*$ of those sets with $\varepsilon=1/n$, we obtain that $u=0$ a.e. on $Q_T\setminus O_a$. 

Now let us check  that $u$ satisfies the equation in the weak sense. Let $v$ be any test function in $L^2(0,T ; H^1_0(\O))$ such that $v=0$ a.e. in $Q_T \setminus O_a$. Then $a u_\l v =0$ a.e. in $Q_T$, and using the fact that $u_\l$ is a solution of $(P_\l)$, we can write 
$$\langle u'_\l , v \rangle_{L^2(Q_T)} + \langle \nabla u_\l ,  \nabla v \rangle_{L^2(Q_T)}    = \langle f_\l,v\rangle_{L^2(Q_T)}.$$ 
Thus passing to the (weak) limit in $u_\l$, $u'_\l$ and $f_\l$ we get 
$$\langle u' , v \rangle_{L^2(Q_T)} + \langle \nabla u ,  \nabla v \rangle_{L^2(Q_T)}    = \langle f,v\rangle_{L^2(Q_T)}.$$

To conclude that $u$ is a solution of $(P_\infty)$ it remains to prove that $u(x,0)=g(x)$ for a.e. $x\in \Omega$. For this purpose, we let $v \in C^1([0,T], H^1_0(\Omega))$ be any function satisfying $v(T)=0$. Testing the equation with this $v$, using that $u_\l(0)=g_\l$ and integrating   by parts with respect to $t$ we obtain

$$-\langle g_\l, v(0)\rangle_{L^2(\O)} -\int_{0}^T\langle u_\l, v'\rangle_{L^2(\O)} + \int_0^T \langle \nabla u_\l , \nabla v \rangle_{L^2(\O)} =  \int_0^T \langle  f_\l , v\rangle.$$
Passing to the limit in $\l$ and using the weak convergence of $g_\l$ to $g$, we get 
$$-\langle g, v(0)\rangle_{L^2(\O)} -\int_{0}^T\langle u, v'\rangle_{L^2(\O)} + \int_0^T \langle \nabla u , \nabla v \rangle_{L^2(\O)} =  \int_0^T \langle f , v\rangle_{L^2(\O)}.$$
Integrating back again by parts on $u$ yields
$$\langle g, v(0)\rangle_{L^2(\O)}  = \langle u(0), v(0)\rangle_{L^2(\O)},$$
and since $v(0)$ is arbitrary, we deduce that $u(0)=g$ in $L^2(\O)$.

Finally, the convergence of $u_\l$ to $u$ holds a  priori up to a subsequence, but by uniqueness of the solution for the problem $(P_\infty)$ (see Proposition \ref{unique}), the convergence holds for the whole sequence.
\end{proof}

\begin{corollary} \label{existence}Let  $O_a\subset Q_T$ be open and increasing in time (in the sense of Remark \ref{Increasing-in-time}), and let $f\in L^2(\O\times (0,T)) $ and $g \in H^1_0(\O)$. Then there exists a (unique) solution for $(P_\infty)$. 
\end{corollary}

\begin{proof} It suffices to apply Proposition \ref{weakConv} with, for instance $a(x,t):= {\rm dist}((x,t),\overline{O_a})$, $f_\l=f$ and $g_\l=g$. 
\end{proof}

%%%%%%%%%%%%%%%%%%%%%%%%%%%%%%

\begin{remark}{\rm (Convergence in $\mathcal{D}'(\Omega\times (0,T))$)}. Under \emph{Assumption (A)}, letting $u$ being the weak limit of $u_\lambda$ in $L^2(Q_T)$, we already know  that 
$$u=0\quad \hbox{a.e. in}\quad Q_T \setminus O_a.$$
 Then 
$$f_\lambda+\Delta u_\lambda -u'_\lambda  \longrightarrow f+\Delta u -u' \text{ in }\mathcal{D}'(\Omega\times (0,T)),$$
which implies that 
\begin{eqnarray}
\lambda a u_\lambda  \longrightarrow h \text{ in }\mathcal{D}'(\Omega\times (0,T)), \label{Conv0}
\end{eqnarray}
for some distribution $h=f+\Delta u-u'  \in \mathcal{D}'(\Omega \times (0,T))$,   supported in $O_a^c$. Actually, since $u=0$ in $O^c_a$, we have 
$$\Delta u=0 \quad \text{ and }\quad  u'=0 \quad \text{ in } \mathcal{D'}(\overline{O_a}^c).$$ 
This means that 
$$h=0 \quad \text{ in }\mathcal{D'}(O_a) \quad \text{ and }\quad  h=f  \text{ in }\mathcal{D'}(\overline{O_a}^c). $$
Notice that, \emph{a priori}, $h$ could have a singular part supported on  $\partial O_a$. We finally deduce that 
 \begin{eqnarray}
\lambda a u_\lambda \underset{\lambda \to +\infty}{\longrightarrow}  f{\bf 1}_{O_a^c} +\Delta u|_{\partial O_a} \text{ in } \mathcal{D}'(\Omega \times (0,T)). \label{distribConv}
 \end{eqnarray}

\end{remark}

%%%%%%%%%%%%%%%%%%%%%%%%%%%%%%

\subsection{Strong convergence in $L^2(0,T;H^1(\O))$}

We now go further  using the same argument as for the stationary problem, and prove a stronger convergence which is one of our main results.

\begin{theorem} \label{mainM}Under the same hypotheses as in Proposition \ref{weakConv}, denote by $u_\lambda$ the solution of $(P_\l)$. Then, $u_\l$ converges strongly  in $L^2(0,T;H^1_0(\O))$ to the solution $u$ of problem $(P_\infty)$.
\end{theorem}
\begin{proof} We already have the bound
$$\|u_\l\|_{L^2(0,T;H^1(\O))}\leq C,$$
and we already know (by Proposition \ref{weakConv}) that  $u_\l$ converges  weakly in $L^2(0,T;H^1(\O))$ to $u$, the unique solution of problem $(P_\infty)$.

Moreover, by the lower semicontinuity of the norm with respect to the weak convergence, there holds
$$\|u\|_{L^2(0,T;H^1(\O))}\leq \liminf_{\lambda}\|u_\l\|_{L^2(0,T;H^1(\O))}.$$
Hence, to prove the strong convergence we only need to prove the reverse inequality, with a limsup. For this purpose we use the fact  that $u(t)$ is a competitor for $u_\lambda(t)$ in the minimization problem solved by $u_\l$ at $t$ fixed. Indeed, for a.e. $t$ fixed, $u_\l$ solves 
$$-\Delta u_\l +\l a u_\l = f_\l -u_\l',$$
thus, $u_\l$ is a minimizer  in $H^1_0(\Omega)$ for the energy 
$$v\mapsto E_\l(v) - 2\int_{\Omega } v(f_\l - u'_\l),$$
where $E_\l$ is defined by \eqref{Def:Elambda}.
Furthermore, due to the bound \eqref{derbound} obtained in Lemma\;\ref{lem2}, since $f_\l$ is bounded in $L^ 2(Q_T)$ and $g_\lambda$ is bounded in 
$L^ 2(\O)$ and satisfies \eqref{Bound:glambda}, 
we know that $u_\l'$ is bounded in $L^2(Q_T)$ , and 
$$u_\l'\to u'\quad \hbox{weakly in}\quad L^2(Q_T).$$ 
We also know that, up to a subsequence, $u_\l \to u$ strongly in $L^2(Q_T)$ (because it is bounded in $H^1(Q_T)$).

 Now, using that $u$ is a competitor for $u_\l$ (for a.e. $t$ fixed), we obtain
\begin{eqnarray}
\int_{\Omega} |\nabla u_\l |^2 dx - 2\int_{\Omega } u_\l(f_\l - u'_\l)&\leq&  E_\l(u_\l) - 2\int_{\Omega } u_\l(f_\l - u'_\l)\leq E_\l(u) - 2\int_{\Omega } u(f_\l - u'_\l) \notag \\
&  =  &\int_{\Omega} |\nabla u|^2 dx - 2\int_{\Omega } u(f_\l - u'_\l).\notag
\end{eqnarray}
Integrating in $t\in [0,T]$, passing to the limsup in $\l$ and since we have the convergence 
$$\int_{Q_T} u_\l(f_\l - u'_\l) \rightarrow \int_{Q_T} u(f - u'),$$
we get the desired inequality, which  achieves the proof.  
\end{proof}

%%%%%%%%%%%%%%%%%%%%%%%%%%%%%%%%%%%

%%%%%%%%%%%%%%%%%%%%%%%%%%%%%%%%%%%%%
\section{Simon's exponential estimate}
\subsection{The stationary case}
\label{sectionSimon}
Following a similar argument to \cite[Theorem 4.1]{SimonII} we ascertain some strong convergence far from the set $\Omega_a:={\rm Int}(K_a)$, where $K_a$ is defined by $K_a:=\{x \in \overline{\Omega} ; a(x)=0 \}$).

\begin{lemma}\label{simon} Let $a:\overline{\Omega}\to \R^+$ be a continuous non-negative potential and $u_\lambda$ be the unique weak solution in $H^1_0(\Omega)$ of  
$-\Delta u_\l+\lambda a u_\l=f_\l$ in $\Omega$. 
Assume that $\Omega_a:={\rm Int}\{a(x)=0\}={\rm Int}\{K_a\}$ is non empty (hypothesis \eqref{nontriviality}).
Let $\varepsilon >0$ be fixed, and define
$$\Omega_\varepsilon:=\{x \in \Omega ;\, dist(x,\Omega_a )>\varepsilon\} \quad \text{ and } \quad \delta := \min_{x \in \overline{\Omega}_\varepsilon}a(x)>0.$$
Then, there exists a constant $C>0$ such that for all $\lambda >0$ and for all Lipschitz function $\eta:\O\rightarrow \R$ that is equal to $1$ in $\Omega_{2\varepsilon}$ and to $0$ outside $\Omega_{\varepsilon}$, we have
\begin{eqnarray}
\int_{\Omega_\varepsilon} e^{2\sqrt{\l \frac{\delta}{2}}{\rm dist}(x,\Omega_{2\varepsilon}^c)} \eta^2  u_\l\left(\frac{\lambda \delta}{2}u_\l-f_\l\right) dx \leq  C.\label{ineq3}
\end{eqnarray}

\end{lemma}

\begin{proof} To lighten the notation, in this proof we will denote by $\langle \cdot,\cdot\rangle$ the scalar product in $L^2(\O)$.

  Let $\varepsilon>0$ be fixed. For any function $\psi$ supported in $\Omega_\varepsilon$ and for any Lipschitz function 
  $\rho:\Omega \to \R$ satisfying $|\nabla \rho|^2\leq \delta/2$ and defined as the limit of a sequence $\{\rho_n\}$ of $C^\infty(\O)$ functions, a direct calculation shows the following identities \\
$$\Delta(e^{-\sqrt{\lambda}\rho})=-\sqrt{\lambda}e^{-\sqrt{\lambda}\rho}\Delta \rho+\lambda |\nabla \rho|^2 e^{-\sqrt{\lambda}\rho},$$
\begin{eqnarray}
\Delta(e^{-\sqrt{\lambda}\rho} \psi) &=& \psi \Delta (e^{-\sqrt{\lambda}\rho}) + e^{-\sqrt{\lambda}\rho} \Delta \psi + 2\nabla (e^{-\sqrt{\lambda}\rho}) \cdot \nabla \psi, \notag 
\end{eqnarray}
from which we obtain that\\
$$\langle e^{\sqrt{\lambda}\rho} \psi, (-\Delta + \lambda a ) e^{-\sqrt{\lambda}\rho}\psi \rangle  =  
\langle \psi, ( \sqrt{\lambda} \Delta \rho - \lambda|\nabla \rho|^2+2\sqrt{\lambda}\nabla \rho \cdot\nabla-\Delta +\lambda a)\psi \rangle.$$
Then, expanding $\Delta(\rho \psi)$ and rearranging terms as follows:
$$
\Delta(\rho \psi)  = \Delta \rho \psi+2 \nabla \rho \cdot \nabla \psi +\Delta \psi \rho
=  \Delta \rho \psi+ \nabla \rho \cdot \nabla \psi + \nabla \cdot(\rho \nabla \psi),
$$
it is not difficult to see that
$$
\langle \psi, (\Delta \rho+2 \nabla \rho\cdot\nabla)\psi \rangle  = \langle \psi, \Delta (\rho\psi) \rangle - \langle \psi, \rho \Delta \psi \rangle
= \langle \psi, \Delta (\rho\psi) \rangle - \langle \rho\psi,\Delta \psi \rangle = 0,
$$
since $\psi$ is compactly supported in $\O$.  
Therefore, by the positivity of $-\Delta$ and using the assumption $|\nabla \rho|^2\leq \delta/2 $ we get
\begin{eqnarray}
\langle e^{\sqrt{\lambda}\rho} \psi, (-\Delta + \lambda a ) e^{-\sqrt{\lambda}\rho}\psi \rangle
&\geq& \langle \psi, [\lambda(a-|\nabla \rho|^2)] \psi \rangle \notag \\
&\geq & \frac{\l \delta}{2} \langle \psi,   \psi \rangle =\frac{\lambda \delta}{2} \|\psi\|^2_2. \label{ineq1}
\end{eqnarray}
Next, we apply \eqref{ineq1} with the choice $\psi=e^{\sqrt{\l}\rho} \eta u_\l(x)$, where $\eta$ is a function equal to $1$ in $\Omega_{2\varepsilon}$ and $0$ outside $\Omega_{\varepsilon}$. 
Since by construction $\|u_\l\|_2+\|\nabla u_\l\|_2\leq C \|f_\l\|_2$, \eqref{ineq1} implies
\begin{eqnarray}
\frac{\lambda \delta}{2}\int_{\Omega_\varepsilon} e^{2\sqrt{\l}\rho} \eta^2  u_\l^2 dx&=&\frac{\lambda \delta}{2}\|\psi\|_2^2 \leq  \int_{\Omega}  
e^{2\sqrt{\lambda}\rho} \eta u_\l (-\Delta + \lambda a  ) (\eta u_\l ) dx \notag \\
&\leq & \int_{\Omega_\varepsilon \setminus \Omega_{2\varepsilon}}  e^{2\sqrt{\lambda}\rho} u_\l (-u_\l\Delta \eta - 2\nabla \eta \cdot \nabla u_\l) dx  \notag\\
&& \quad \quad \quad \quad \quad  +  \int_{\Omega_\epsilon}  e^{2\sqrt{\lambda}\rho} \eta^2 u_\l (-\Delta u_\l+\lambda a u_\l) dx \notag \\
&\leq & \int_{\Omega_\varepsilon \setminus \Omega_{2\varepsilon}}  e^{2\sqrt{\lambda}\rho} u_\l (-u_\l\Delta \eta - 2\nabla \eta \cdot \nabla u_\l) dx +  \int_{\Omega_\epsilon}  e^{2\sqrt{\lambda}\rho} \eta^2 u_\l  f_\l dx \notag \\
&\leq & C\|f_\l\|_2\;e^{2\sqrt{\lambda}M}+ \int_{\Omega_\epsilon}  e^{2\sqrt{\lambda}\rho} \eta^2 u_\l  f_\l dx, \label{ineq2}
\end{eqnarray}
where 
$$M :=\sup_{x \in \Omega \setminus \Omega_{2\varepsilon}} \rho(x),$$
and the constant $C\equiv C(\Delta \eta, \nabla \eta,\epsilon)$ in \eqref{ineq2} depends on the derivatives of $\eta$ and $\varepsilon$. 
Now we take the particular choice 
$$\rho(x):= \sqrt{\frac{\delta}{2}} {\rm dist}(x,\Omega_{2\varepsilon}^c),$$ 
which satisfies all our needed assumptions (i.e. $\rho$ is Lipschitz with $|\nabla \rho|^2\leq \delta/2$ and supported in $\Omega_{\varepsilon}$). In this case $M=0$ thus \eqref{ineq2}  simply implies
$$
\frac{\lambda \delta}{2}\int_{\Omega_\varepsilon} e^{2\sqrt{\l}\rho} \eta^2  u_\l^2 dx \leq  C+ \int_{\Omega_\epsilon}  e^{2\sqrt{\lambda}\rho} \eta^2 u_\l   f_\l dx,
$$
or differently,
$$
\int_{\Omega_\varepsilon} e^{2\sqrt{\l}\rho} \eta^2  u_\l\left(\frac{\l \delta}{2}u_\l-f_\l\right) dx \leq  C,
$$
which ends the proof.
\end{proof}

\begin{remark} The previous lemma can be used for instance in the following two particular cases: first in the particular case when $f=0$ in $\Omega \setminus \Omega_a$. Thus, we get the useful rate of convergence of $u_\l \to 0$ as $\lambda\to 0$ far from $\Omega_a$:
$$
\int_{\Omega_{2\varepsilon}} \lambda e^{2\sqrt{\l \frac{\delta}{2}} {\rm dist}(x,\Omega_{2\varepsilon}^c)}   u_\l^2 dx \leq  C.
$$
This is much better compared to the usual and simple energy bound:
$$\lambda \int_{\Omega} a u_\l^2 \leq C.$$

Another application is when $u_\lambda$ is an eigenfunction (this is actually the original framework of Simon {\rm \cite{SimonII}}), i.e. 
when $f_\l=\sigma(\lambda) u_\l$ and with $\s(\l)$ standing for the first eigenvalue associated with $u_\l$. 
In this case, since we are assuming that the potential $a$ might vanish in a subdomain (it could vanish at a single point, as performed by Simon \cite{SimonII}), we have that $\sigma(\lambda)$ is bounded (cf. \cite{PJ2} for further details).  Consequently, thanks to this bound for $\l$ large enough $\frac{\lambda \delta}{2} -\sigma(\lambda)\geq 1$ which implies 
$$
\int_{\Omega_{2\varepsilon}} e^{2\sqrt{\l \frac{\delta}{2}} {\rm dist}(x,\Omega_{2\varepsilon}^c)}   u_\l^2 dx \leq  C.
$$
\end{remark}

%%%%%%%%%%%%%%%%%%%%%%%%%%%%%%%%
\subsection{The parabolic case}
%%%
We now extend the previous decay estimate  to the parabolic problem.

\begin{lemma}\label{simon2} Let $a:\overline{Q}_T \to \R^+$ be a continuous non-negative potential such that $O_a$ is non empty, $f_\lambda \in L^2(Q_T)$, $g_\lambda \in H^1_0(\overline{O_a} \cap \{t=0\})$ and let $u_\lambda$ be the   solution of  $(P_\lambda)$.

For every $\varepsilon >0$, we define
\be
A_\varepsilon:=\big\{(x,t) \in \overline{Q}_T ; dist\big((x,t), O_a \big)>\varepsilon \big\} \quad \text{ and } \quad \delta := \min_{(x,t) \in \overline{A_\varepsilon}}a(x,t)>0.\label{Def:Aeps-delta}
\ee
Then, for any $\lambda \geq 4$, and for any Lipschitz function $\eta:Q_T\rightarrow \R$   equal to $1$ in $A_{2\varepsilon}$ and $0$ outside $A_{\varepsilon}$, there exists a constant $C>0$ such that
$$
\int_{A_\varepsilon} e^{2\sqrt{\l \frac{\delta}{2}} {\rm dist}((x,t),A_{2\varepsilon}^c)} \eta^2  u_\l(x)\left(\frac{\lambda \delta}{2}u_\l-f_\l\right) dx \, dt\leq  C.
$$

\end{lemma}

\begin{proof}  Let $\varepsilon>0$ be fixed. We consider any function $\psi$ supported in $A_\varepsilon$ and  any Lipschitz function $\rho:Q_T \to \R$ satisfying $|\partial_t\rho|+|\nabla \rho|^2\leq \delta/2$. Then we notice that for any fixed $t$, the functions $x\mapsto \rho(x,t)$, $x\mapsto \psi(x,t)$ and $x\mapsto a(x,t)$ satisfy all the conditions required to prove the key estimate \eqref{ineq1}  in the stationary case. Thus, for any fixed $t$, there holds
$$
\langle e^{\sqrt{\lambda}\rho} \psi, (-\Delta + \lambda a ) e^{-\sqrt{\lambda}\rho}\psi \rangle_{L^2(\Omega)}
\geq \frac{\lambda \delta}{2} \|\psi\|^2_{L^2(\Omega)}.
$$
Integrating this inequality over $t\in (0,T)$ yields
$$
\langle e^{\sqrt{\lambda}\rho} \psi, (-\Delta + \lambda a ) e^{-\sqrt{\lambda}\rho}\psi \rangle_{L^2(Q_T)}
\geq \frac{\lambda \delta}{2} \|\psi\|^2_{L^2(Q_T)}.
$$
Next we also compute 
$$\partial_t (e^{-\sqrt{\lambda}\rho}\psi)=-\sqrt{\lambda}e^{-\sqrt{\lambda}\rho}\psi  \partial_t \rho +e^{-\sqrt{\lambda}\rho}\partial_t \psi,$$
so that,

\begin{eqnarray}
\langle e^{\sqrt{\lambda}\rho} \psi, \partial_t (e^{-\sqrt{\lambda}\rho}\psi) \rangle_{L^2(Q_T)} &=&-\sqrt{\lambda} \langle  \psi, \psi  \partial_t \rho \rangle_{L^2(Q_T)} + \langle   \psi,  \partial_t \psi \rangle_{L^2(Q_T)} \notag \\
&=&  -\sqrt{\lambda} \langle  \psi, \psi  \partial_t \rho \rangle_{L^2(Q_T)} +\frac{1}{2} \|\psi(T)\|_{L^2(\Omega)}^2-\frac{1}{2} \|\psi(0)\|_{L^2(\Omega)}^2\notag \\
&\geq & -\sqrt{\lambda} \frac{\delta}{2} \|\psi\|^2_{L^2(Q_T)} -\frac{1}{2}\|\psi(0)\|^2_{L^2(\Omega)}.\notag 
\end{eqnarray}

Gathering the previous estimates, we obtain 
\begin{eqnarray}
\langle e^{\sqrt{\lambda}\rho} \psi, (\partial_t-\Delta + \lambda a ) e^{-\sqrt{\lambda}\rho}\psi \rangle_{L^2(Q_T)} \geq \frac{(\lambda -\sqrt{\lambda})\delta}{2} \|\psi\|^2_2 -\frac{1}{2}\|\psi(0)\|^2_{L^2(\Omega)}. \label{inequalityP}
\end{eqnarray}

Next, we apply \eqref{inequalityP} with the choice $\psi=e^{\sqrt{\l}\rho} \eta u_\l(x)$, where $\eta$ is a function equal to $1$ in $A_{2\varepsilon}$ and $0$ outside $A_{\varepsilon}$. We assume that $\lambda \geq 4$   so that $\lambda -\sqrt{\lambda}\geq \lambda/2$. Due to the assumptions, $g\in H^1_0(\overline{O_a}\cap \{t=0\})$ and, then, $\|\psi(0)\|_{L^2(\Omega)}=0$. Thus, \eqref{inequalityP} implies
\begin{eqnarray}
\frac{\lambda \delta}{4}\int_{A_\varepsilon} e^{2\sqrt{\l}\rho} \eta^2  u_\l^2 dx\,  dt&=&\frac{\lambda \delta}{4}\|\psi\|_2^2 \leq 
 \int_{\Omega\times (0,T)}  e^{2\sqrt{\lambda}\rho} \eta u_\l (\partial_t-\Delta + \lambda a  ) (\eta u_\l ) dx \, dt \notag \\
&\leq&  \int_{A_\varepsilon \setminus A_{2\varepsilon}}  e^{2\sqrt{\lambda}\rho} \eta u_\l (u_\l\partial_t\eta-u_\l\Delta \eta - 2\nabla \eta \cdot \nabla u_\l) dx \, dt\notag \\
&\text{ }& \quad \quad \quad  + \int_{A_\epsilon}  e^{2\sqrt{\lambda}\rho} \eta^2 u_\l (\partial_tu_\l-\Delta u_\l+\lambda a u_\l) dx \, dt \notag \\
&=& \int_{A_\varepsilon \setminus A_{2\varepsilon}}  e^{2\sqrt{\lambda}\rho} u_\l (u_\l\partial_t\eta-u_\l\Delta \eta - 2\nabla \eta \cdot \nabla u_\l) dx \,  dt\notag \\
&\text{ }& \quad \quad \quad   +  \int_{A_\epsilon}  e^{2\sqrt{\lambda}\rho} \eta^2 u_\l  f_\l dx \, dt \notag \\
&\leq & C \|u_\l\|_{L^2(0,T;H_0^1(\Omega))}^2\, e^{2\sqrt{\lambda}M}
 +  \int_{A_\varepsilon}  e^{2\sqrt{\lambda}\rho} \eta^2 u_\l  f_\l dx \, dt, \label{ineq234}
\end{eqnarray}
where 
$$M :=\sup_{x \in Q_T \setminus A_{2\varepsilon}}  \rho(x,t),$$
and the constant $C\equiv C(\partial_t \eta, \Delta \eta, \nabla \eta,\epsilon)$  in \eqref{ineq234} depends on the derivatives of $\eta$ and on $\varepsilon$. 
Now we take the particular choice 
$$\rho(x,t):= \sqrt{\frac{\delta}{2}} {\rm dist}((x,t),A_{2\varepsilon}^c),$$ 
which satisfies all our needed assumptions (i.e. $\rho$ is Lipschitz with $|\partial_t \rho|+|\nabla \rho|^2\leq \delta/2$ and supported in $A_{\varepsilon}$). In this case, $M=0$ so that that \eqref{ineq234} reduces to
$$
\frac{\lambda \delta}{4}\int_{A_\varepsilon} e^{2\sqrt{\l}\rho} \eta^2  u_\l^2 dx \, dt \leq  C+ \int_{A_\varepsilon}  e^{2\sqrt{\lambda}\rho} \eta^2 u_\l   f_\l dx \, dt,
$$
and hence
$$
\int_{A_\varepsilon} e^{2\sqrt{\l}\rho} \eta^2  u_\l(x)\left(\frac{\l \delta}{4}u_\l-f_\l\right) dx \leq  C.
$$
\end{proof}

We end this section by noticing that Theorem \ref{main2} follows directly from Lemma \ref{simon2}.

\begin{corollary}   In the particular case when $f=0$ in $Q_T \setminus O_a$ we get the useful rate of convergence of $u_\l \to 0$ as $\lambda\to \infty$ far from $O_a$:
$$
 \lambda e^{4\varepsilon \sqrt{\l} \frac{\delta}{2}}  \int_{A_{2\varepsilon}}  u_\l^2 dx \, dt \leq  C.
$$
\end{corollary}

%%%%%%%%%%%%%%%%%%%%%%%%%%%%%%%%%%%%%%%%%%%%%%%%%%%%%%%%%%%%%%%%

\end{document}